\documentclass[preprint,12pt]{elsarticle}
\usepackage{amsmath,amssymb,amsthm,mathrsfs,dsfont}
\usepackage[margin=3cm]{geometry}
\usepackage{titlesec,hyperref}
\usepackage{color}
\usepackage{fancyhdr}
\titleformat{\subsection}{\it}{\thesubsection.\enspace}{1pt}{}
\makeatletter
\def\ps@pprintTitle{%
   \let\@oddhead\@empty
   \let\@evenhead\@empty
   \let\@oddfoot\@empty
   \let\@evenfoot\@oddfoot
}
\makeatother

\newtheorem {theorem}{Theorem}[section]
\newtheorem {lemma}[theorem]{Lemma}

\newtheorem {remark}[theorem]{Remark}

\newtheorem {definition}[theorem]{Definition}

\newcommand{\di}{\mbox{div}}

\newcommand{\R}{\mathbb{R}}
\newcommand{\p}{\partial}
\def\u{\mathbf{u}}

\def\w{\mathbf{w}}

\def\O{\Omega}

\allowdisplaybreaks

\begin{document}
\begin{frontmatter}
\title{Non-existence of global classical solutions to barotropic compressible Navier-Stokes equations with degenerate viscosity and vacuum}
\author[add1]{Minling Li}\ead{limling3@mail2.sysu.edu.cn}
\author[add1]{Zheng-an Yao}\ead{mcsyao@mail.sysu.edu.cn}
\author[add1]{Rongfeng Yu\corref{cor1}}\ead{yurongf@mail.sysu.edu.cn}\cortext[cor1]{Corresponding author}
\address[add1]{School of Mathematics, Sun Yat-sen University, Guangzhou, 510275, CHINA}

\begin{abstract}
We are concerned about the barotropic compressible Navier-Stokes equations with density-dependent viscosities which may degenerate in vacuum. We show that any classical solution to barotropic compressible Navier-Stokes equations in bounded domains will blow up, when the initial density admits an isolated mass group and the viscousity coefficients satisfy some conditions. A new condition on viscosities is first put forward in this paper.
\end{abstract}


\end{frontmatter}
\vspace*{10pt}

\section{Introduction}
This paper is concerned with the barotropic compressible Navier-Stokes system:  
\begin{equation} \label{eq1}
\left\{\begin{array}{lr} 
\rho_t +\di(\rho \u)=0,\\
(\rho \u)_t+\di(\rho \u \otimes \u)-\di(2\mu(\rho)D(\u))-\nabla (\lambda(\rho)\di \u)+ \nabla P= 0,
\end{array}\right.
\end{equation}
with the initial data
\begin{equation}\label{initial data}
\rho(x,0)=\rho_{0}(x), \u(x,0)=\u_{0}(x).
\end{equation}
Here $x\in\R^d (d\geq 2)$ is the spatial coordinate, $t\geq 0$ is time. $\rho \geq 0$, $\u$ and $P$ denote the fluid density, velocity and pressure, respectively. $D(\u)$ is the strain tensor with the form $$D(\u)=\dfrac{(\nabla \u+\nabla \u^\top)}{2}.$$
$\mu(\rho)$ and $\lambda(\rho)$ are the viscosity coefficients satisfing the physical restrictions:
\begin{equation}\label{viscostiy0}
\mu(\rho)\geq 0,\; 2\mu(\rho)+d\lambda(\rho)\geq 0.
\end{equation}
We consider the polytropic gases for which the equation of state is given by 
\begin{equation}\label{pressure}
P(\rho)=A\rho^{\gamma},
\end{equation}
where $A$ is a positive constant, setting to be unity for convenience, $\gamma>1$ is the adiabatic exponent. 

A great number of mathematicians have made great efforts and achieved fruitful results for the local and global existence of solutions to compressible Navier-Stokes equations. The one-dimensional problem has been studied extensively, see \cite{2011Ding,Liliang2016,li2008vanishing,Qinyao2010,yang2002compressible,Yangyao2001} and references cited therein. In multi-dimensional case, Matsumura and Nishida proved the global existence of classical solutions for compressible Navier-Stokes equations in \cite{matsumura1983initial-boundary}, where they required that the initial data close to a equilibrium state far away from vacuum. When considering general data, one has to face the possible appearance of vacuum. However, as observed in \cite{xin1998blowup,xin2013on}, the smooth solutions to the full compressible Navier-Stokes equations with constant viscosities will blow up in finite time. Some other related results can be found
in \cite{bian2019finite,cho2006blow-up,2019liwang,rozanova2008blow} and the references therein. When it comes to barotropic compressible Navier-Stokes equations with constant viscosity coefficients, Lions  made a breakthrough in \cite{Lion}, where he proved the global existence of weak solutions for any initial data containing vacuum, as soon as the initial energy was finite and $\gamma\geq \gamma_n (\gamma_n=\dfrac{3n}{n+2}, n=2,3)$. Jiang and Zhang extended Lion's result to $\gamma>1$ for spherically symmetric initial data in \cite{jiang2001on}.  Feireisl, Novotn\'y and Petzeltova \cite{feireisl2001on} improved the condition to $\gamma >3/2$ in three-dimensional space. Huang, Li and Xin showed the global existence of classical solutions to three-dimensional barotropic compressible Navier-Stokes equations for initial data with small total energy but possible large oscillations and containing vacuum states in \cite{huang2012global}. Later, Li and Xin \cite{li2019global} also proved the global existence in two-dimensional space.

For the case that the viscosity coefficients depend on the density and degenerate at the vacuum, there are more interesting phenomena. Such model was derived from the Boltzmann equations by Liu, Xin and Yang in \cite{liu1998}.  Vaigant and Kazhikhov proved that there exists a unique classical solution for two-dimensional barotropic compressible Navier-Stokes equations with $\mu=const.$, $\lambda=\rho ^\beta (\beta>3)$ in periodic domains when the initial density away from vacuum in \cite{1995Kazhikhov}. Later, Huang and Li \cite{2016Huang} established the global existence of classical solutions to Cauchy problem of this model for $\beta>4/3$ with vacuum at far field, while Jiu, Wang and Xin \cite{jiu2014global,jiu2018} obtaining the similar result with non-vacuum at far field. 

When the viscosity coefficients are both dependent of density, Bresch and Desjardins \cite{bresch2003existence} proposed a new entropy inequality (BD-entropy) under an additional constraint on the viscosity coefficients, which played an important role in proving the existence of weak solutions. Based on this conclusion, there are some results about weak solution for compressible Navier-Stokes equations. Bresch, Desjardins and G\'{e}rard-Varet showed the existence of global weak solutions for compressible Navier-Stokes equations with an additional appropriate constraints on the initial density profile and domain curvature in \cite{bresch2007on}. Guo, Jiu and Xin proved the global existence of weak solutions for the spherically symmetric initial data in \cite{guo2008spherically}. Li and Xin proved the global existence of weak solutions for two-dimensional and three-dimensional Cauchy problem of barotropic compressible Navier-Stokes equations in \cite{jing2015global}. In the same time, Vasseur and Yu gave the proof of the global existence of weak solutions for three-dimensional Navier-Stokes equations by using a different method in \cite{vasseur2016existence}. A nature question is: Can we improve the regularity of the weak solutions above? Li, Pan and Zhu investigated the local existence of regular solutions for compressible barotropic Navier-Stokes equations with density-dependent viscosities in \cite{Li2016Recent,2015On}. Luo and Zhou extended the result in \cite{2019luozhou}. Recently, Xin and Zhu \cite{xin2019global} proved the global-in-time well-posedness of regular solutions for a class of smooth initial data for Cauchy problem. So it is important to investigate the classical solutions for multi-dimensional compressible Navier-Stokes equations with degenerate viscosities whether exist globally.

In this paper, we consider the barotropic compressible Navier-Stokes equations in a bounded smooth domain $\Omega\subset\R^d$ and come to a conclusion that the classical solutions to the initial boundary value problem with density-dependent viscosities satisfying some conditions (i.e. Condition (i) -- (iii) in Section 2) will blow up in finite time if the initial density admits an isolated mass group. The key step of the proof is to handle the term
\begin{equation}\label{term}
\int_{0}^{t}\int_{U{(s)}} (2\mu (\rho)+d\lambda(\rho))dxds. 
\end{equation}
In the case that the viscosity coefficients satisfy Condition (i) or (ii) in Section 2, it is simple to deduce that
\[\int_{0}^{t}\int_{U{(s)}} (2\mu (\rho)+d\lambda(\rho))dxds \leq C, \]
which is inspired by \cite{duan2019finite-time}. In the case that the viscosity coefficients satisfy Condition (iii), we could split $U(t)$ into two parts, i.e. $U(t)=\mathcal{C}(t)\cup U_1(t),$ with $\mathcal{C}(t):=\{x\in U(t)|0\leq\rho(x,t)\leq 1\}$ and $U_1(t):=U(t)\backslash \mathcal{C}(t)$. In $\mathcal{C}(t)$, $\rho$ is bounded, and it is easy to get
\[\int_{0}^{t}\int_{\mathcal{C}{(s)}} (2\mu (\rho)+d\lambda(\rho))dxds \leq Ct. \]
Since $\mu'(\rho)\geq c$ for some positive constant $c$ in $U_1(t)$, we could get higher regularity on the density and then obtain by using interpolation inequality and Sobolev embedding that
\[\int_{0}^{t}\int_{U_1{(s)}} (2\mu (\rho)+d\lambda(\rho))dxds \leq C(t+1), \]
which will be proved carefully in Lemmas \ref{sec} and \ref{kk}. The third case is our main innovation in this paper. Then we denote the quantity $$G(t):=\int_{U{(t)}}\rho|x|^2 dx$$ and conclude by a series of calculations that
\[G(t)\geq G_{0}+G_{1}t+C M_{0}^{\gamma}t^{2}-C(t+1)^{3/2}.\]
On the other hand, from the fact that $U(t)\subset \Omega$ is bounded and the mass equation $\eqref{eq1}_1$ , we have
$$
\begin{aligned}
G(t)\leq M_{0}R^{2},~~~\text{for some}~~R>0.
\end{aligned}
$$
Therefore, we have that $t$ must be finite.

\section{Main Results}
Before presenting our main results, we need to give the following two definitions. First, we define the classical solutions to the initial value boundary problem for System \eqref{eq1}--\eqref{initial data} with suitable boundary conditions.
\begin{definition}
For $T>0$, a pair $(\rho(x,t),\u(x,t)) $ is called a classical solution to the initial value boundary problem for System \eqref{eq1}--\eqref{initial data} with a suitable boundary condition on $\p\O$ if the nonnegative function $\rho \in C^{1}(\O \times [0,T))$, and the vector field $\u \in C^{1}([0,T);C^{2}(\O))$ satisfy the system \eqref{eq1} point-wisely on $\O \times [0,T)$, take on the initial condition \eqref{initial data} continuously, and satisfy the boundary condition continuously.
\end{definition}

Next, we will give the definition of the isolated mass group which is introduced in \cite{xin2013on}. 

\begin{definition}
Let $V$, $U$ be two bounded open subsets of $\O$, and $V\subset U$. The pair $(U,V)$ is called an isolated mass group of initial density $\rho_{0}(x)$ if it holds that
\begin{equation*}
\begin{cases}
V\subset\overline{V}\subset U\subset \overline{U}\subset\O,~~U ~\mbox{is connected},\\
\rho_0(x)=0,~~x\in U\setminus V,\\
\int_{V}\rho_0(x)\mbox{d}x>0.
\end{cases}
\end{equation*}
\end{definition}
Let  $(U,V)$ be an isolated mass group of $\rho_{0}(x)$ in $\O$. 
Set $$
\begin{aligned}
&M_{0}=\int_{V}{\rho_{0}(x)dx},\\
&G_{0}=\int_{V}{\rho_{0}(x)|x|^{2}dx},\\
&G_{1}=2\int_{V}{\rho_{0}(x)\u_{0}(x) \cdot x dx},\\
&E_{0}=\int_{V}{\Big(\frac{1}{2}\rho_{0}(x)|\u_{0}(x)|^{2}+\frac{1}{\gamma-1}\rho_{0}^{\gamma}(x)\Big)}dx,\\
&E_{1}=\int_{V}{\Big(\frac{1}{2}\rho_{0}|\u+2\mu '(\rho_0)\nabla \log \rho_{0}|^{2}+\frac{1}{\gamma-1}\rho_{0}^{\gamma}\Big)dx}.
\end{aligned}
$$
\textbf{Conditions on $\mu(\rho)$ and $ \lambda(\rho)$:}\\
We assume that the viscosity coefficients $\mu(\rho)$ and $ \lambda(\rho)$ have the form 
\begin{equation}\label{viscostiy}
\mu(\rho)=a\rho^{\alpha},~~ \lambda(\rho)=b\rho^{\beta},~\text{with}~ a>0, 2a+db\geq 0,
\end{equation}
and satisfy one of the following conditions: 
\begin{itemize}
\item [(i)]  $\alpha=\beta\in (0,+\infty),~2a+db=0;$
\item [(ii)] $\alpha,\beta\in (0,\gamma];$
\item [(iii)] $\alpha=\beta\in (\gamma,\kappa\gamma], ~2a(\alpha-1)=b,$ with $\kappa=\left\{\begin{array}{lr} 
{\rm{for\; any}}\; \kappa \in [1,2),\;{\rm{when}}\; d=2,\\
\frac{d+2}{d},\;~~~~~~~~~~~~~~~~~{\rm{when}}\; d\geq3.\\
\end{array}\right.$
\end{itemize}

\begin{remark}
In general, consider the two finite linear combinations $\sum a_i \rho^{\alpha_i}$ and $\sum b_i \rho^{\beta_i},$ if $\alpha_i=\beta_i \in (0,+\infty),~2a_i+db_i=0$, then they are admissible functions for $\mu(\rho),\lambda(\rho)$; if $\alpha_i,\beta_i\in (0,\gamma],$ then they are still admissible functions for $\mu(\rho),\lambda(\rho)$; similarly, if $\alpha_i=\beta_i\in (\gamma,\kappa\gamma], ~2a_i(\alpha_i-1)=b_i,$ the finite linear combinations are also admissible.
\end{remark}


The main result is described as follows:
\begin{theorem}\label{them1}
	Let $(\rho(x,t),\u(x,t))$ be a classical solution to the compressible Navier-Stokes system $\eqref{eq1}$ on $\O \times [0,T)$ with initial data \eqref{initial data}, supplied with a suitable condition on $\p\O$. Suppose that the initial density $\rho_{0}(x)$ admits an isolated mass group $(U,V)$, and $\mu(\rho),\lambda(\rho)$ satisfy \eqref{viscostiy0} and \eqref{viscostiy}, if one of the following conditions holds:
	
	(1) Condition (i) holds, and $M_0$ is finite;

	(2) Condition (ii) holds, and $M_0$, $E_0$ are finite;
	
	(3) Condition (iii) holds, and $M_0$, $E_0$, $E_1$ are finite.
		
	Then the classical solution $(\rho(x,t),\u(x,t))$ will blow up in finite time.
\end{theorem}

A few remarks are in order:
\begin{remark}
In \cite{bresch2007on}, Bresch et al. showed how to obtain the existence of global weak solutions for both Dirichlet and Navier boundary conditions on the velocity in bounded smooth domains.	Theorem \ref{them1} holds true for these two classical boundary conditions, too. In fact, our result is valid for any physical boundary condition as soon as there is an isolated mass group of the initial density.
\end{remark} 

\begin{remark}
Condition (i) implies that the viscosities satisfy $2\mu(\rho)+d\lambda(\rho)=0$, which is indicated by the kinetic theory of monatomic gases (see \cite{Lion1,yu}). In fact, Theorem \ref{them1} is still true for monatomic gases with the viscosities vanishing at vacuum, since we do not need to handle the term \eqref{term} in this case.  If we replace \eqref{viscostiy} and Condition (ii) by $\mu(\rho),\lambda(\rho) \lesssim 1+\rho^{\gamma}$ with $\mu(\rho),\lambda(\rho)$ vanishing at vacuum, then Theorem \ref{them1} is also valid, which implies the result in \cite{duan2019finite-time} for two-dimensional case.

Furthermore, we could give a more general condition: Continuous functions $\mu(\rho)$ and $\lambda(\rho)$ vanish at vacuum, and there exist $\mu_i(\rho),\lambda_i(\rho)~(i=1,2)$ such that $\mu(\rho)=\mu_1(\rho)+\mu_2(\rho),~\lambda(\rho)=\lambda_1(\rho)+\lambda_2(\rho)$ with
$\mu_1(\rho),\lambda_1(\rho) \lesssim 1+\rho^{\gamma}$ and $2\mu_2(\rho)+d\lambda_2(\rho)=0.$
\end{remark}

\begin{remark}
Condition (iii) implies that $\lambda(\rho)=2\rho\mu'(\rho)-2\mu(\rho)$, which was always assumed in proving the global existence of weak solutions for compressible barotropic Navier-Stokes equations with degenerate viscosity (see \cite{bresch2007on, guo2008spherically, jing2015global, vasseur2016existence}). In fact, Theorem \ref{them1} is also valid if \eqref{viscostiy} and Conditon (iii) are replaced by the following conditions
	$$\lambda(\rho)=2 \rho \mu'(\rho)-2\mu(\rho),~\mu(\rho)\sim\rho^\alpha ~(\alpha>1)~\rm{when}~ \rho< 1,$$
$$ \mu'(\rho)\geq 0,~\mu'(\rho)\geq\epsilon_0>0~\rm{when}~\rho\geq 1,~~~~~~~~~~~~~~~~~~~~~~$$
$$	\mu(\rho) ,\lambda(\rho)\lesssim\left\{\begin{array}{lr} 
	1+\rho^{r \gamma},\;{\rm{for\; all}}\; r \in [1,2),\;{\rm{when}}\; d=2,\\
	1+\rho ^{\frac{d+2}{d}\gamma},\;{\rm{when}}\; d\geq3.\\
	\end{array}\right.$$
	
In \cite{bresch2007on,mellet2007on}, the assumption that $``\mu'(\rho)\geq \epsilon_0>0,~\text{for all} ~ \rho"$ is crucial in the global existence of weak solutions for compressible barotropic Navier-Stokes equations with degenerate viscosity. However, in this paper, we only assume that $\mu'(\rho)\geq\epsilon_0>0$ for $\rho\geq 1$, and $\mu'(0)=0$ is admissible.
	
\end{remark}

\begin{remark}
	In this paper, there is no assumption on the initial velocity, we only need the initial density to have an isolated mass group in bounded domains. However, this is different from Li, Pan and Zhu's blow-up result \cite{2015On} in whole space, where they need the initial velocity to be a constant on $\p U$ additionally.
\end{remark}


\section{Proof of Theorem \ref{them1}}\label{3}
Let $(\rho(x,t),\u(x,t))$ be a classical solution to the barotropic compressible Navier-Stokes system $\eqref{eq1}$ on $\O \times [0,T)$, where $T$ is the
maximal time of existence. Denote by $X(a,t)$ the particle path starting from $a$ when $t=0$, i.e.
\begin{equation}\label{ODE}
\left\{\begin{array}{lr} 
\frac{d}{dt} X(a,t)=\u(X(a,t),t),\\
X(a,0)=a.
\end{array}\right.
\end{equation}
Set
\begin{equation}\label{Ut}
U(t)=\{X(a,t)|a \in U\},\; V(t)=\{X(a,t)|a \in V\}.
\end{equation}
The pair $(U(t),V(t))$ is an isolated mass group of the density $\rho(x,t)$ in $\O$ at time $t$, and it will not disappear for any $t<T$. In fact, $\u \in C^{1}([0,T);C^{2}(\O))$ ensures that $X(a,t)$ is well-defined (existence and uniqueness) by the classical theory of ordinary differential system. Furthermore, $V\subset\overline{V}\subset U\subset \overline{U}\subset\O$ implies that
\begin{equation*}
V(t) \subset \overline{V(t)} \subset U(t)\subset \overline{U(t)}\subset\O.
\end{equation*}

Since $\rho_{0}(x)=0\; \rm{in}\; U \backslash V$, we can immediately obtain the following lemma from mass conservation equation.
\begin{lemma}
	Suppose $(\rho(x,t),\u(x,t))$ is a classical solution to System \eqref{eq1}--\eqref{initial data} with a suitable boundary condition on $\p\O$, it holds that
	\begin{equation}\label{lem2.2}
	\rho =0\; ~~ \rm{in}\; U(t) \backslash V(t).
	\end{equation}
\end{lemma}

Without loss of generality, one can assume that the density $\rho$ and its various derivatives equal to 0 on $\partial U(t)$, since we could otherwise choose $W(t)$ instead of $U(t)$ satisfying $V(t) \subset \overline{V}(t) \subset W(t) \subset \overline{W}(t) \subset U(t)$, then $\rho$ and its various derivatives equal to 0 on $\partial W(t)$. 

\begin{lemma}[Transport formula \cite{Majda}]
	Let $U(t)$ be defined as \eqref{Ut}, for any $f(x,t) \in C^{1}(\mathbb{R}^{d} \times \mathbb{R}^{+})$, we have 
	\begin{equation}\label{lem2.1}
	\frac{d}{dt}\int_{U{(t)}}{f(x,t)dx}=\int_{U(t)}{f_{t}(x,t)dx} +\int_{\partial U(t)}{f(x,t)(\u \cdot \textbf{n})dx},
	\end{equation}
	where $\textbf{n}$ is the unit out normal to $ U(t)$.
	In particular, it holds that
	\begin{equation}\label{lem2.1 k}
	\begin{aligned}
	\frac{d}{dt}\int_{U(t)}{\rho f dx}=\int_{U(t)}{\rho (f_{t}+ \u \cdot \nabla f)dx}.
	\end{aligned}
	\end{equation}
	Here $\rho$ is the density satisfying $\eqref{eq1}_1$.
\end{lemma}

The following lemma is about the standard energy estimates for $(\rho,\u)$.
\begin{lemma}\label{lemk}
	Suppose $(\rho(x,t),\u(x,t))$ is a classical solution to System \eqref{eq1}--\eqref{initial data} in $\O \times [0,T)$ with a suitable boundary condition on $\p\O$, then for each $0<t<T$, we have
	\begin{equation}\label{lem2.3 1}
	\int_{U(t)}{\rho(x,t)dx}=M_{0} >0,
	\end{equation}
	and
	\begin{equation}\label{lem2.3 2}
	\begin{aligned}
	\int_{U(t)}{\Big{(}\frac{1}{2}\rho|\u|^{2} +\frac{1}{\gamma-1}\rho^{\gamma}\Big{)}dx} +\int_{0}^{t}\int_{U(s)}{\Big{(}2\mu(\rho)|D\u|^{2}+ \lambda(\rho)|\di {\u}|^{2}\Big{)}dxds}=E_{0}.
	\end{aligned}
	\end{equation}
\end{lemma}

\begin{proof}
	It follows from transport formula \eqref{lem2.1} that
	$$
	\begin{aligned}
	\frac{d}{dt}\int_{U(t)}{\rho(x,t)dx}=\int_{U(t)}{\rho_{t}dx}+\int_{\partial U(t)}{\rho \u \cdot \textbf{n} dx}=\int_{U(t)}{(\rho_{t}+ \di(\rho \u))dx}=0,
	\end{aligned}
	$$
	which implies \eqref{lem2.3 1}.
	
	Multiplying $\eqref{eq1}_{2}$ by $\u$, and integrating by part lead to
	\begin{equation}\label{lem2.3 21}
	\begin{aligned}
	\frac{1}{2} \frac{d}{dt} \int_{U(t)}{\rho |\u |^{2}dx}+\int_{U(t)}{\big(2\mu(\rho)|D\u|^{2}+ \lambda(\rho)|\di {\u}|^{2}\big)dx}+\int_{U(t)}{\nabla P \cdot \u dx}=0,
	\end{aligned}
	\end{equation}
	where we used the fact that $\mu(0)=\lambda(0)=0$. From the mass equation $\eqref{eq1}_1$, we have
	\begin{equation}\label{lem2.3 23}
	\begin{aligned}
	\int_{U(t)}{\nabla P \cdot \u dx}=&-\int_{U(t)}{\rho^{\gamma} \di {\u}dx}
	=\int_{U(t)}{\rho^{\gamma-1}(\rho_{t}+ \u \cdot \nabla \rho)dx}\\
	=&\frac{1}{\gamma-1} \int_{U(t)}{\rho \big((\rho^{\gamma-1})_{t}+ \u \cdot \nabla \rho ^{\gamma-1}\big)dx}
	=\frac{1}{\gamma-1} \frac{d}{dt} \int_{U(t)}{\rho^{\gamma}dx}.
	\end{aligned}
	\end{equation}
	Substituting \eqref{lem2.3 23} into \eqref{lem2.3 21} leads to
	\begin{equation}\label{lem2.3 25}
	\begin{aligned}
	\frac{d}{dt}\int_{U(t)}{\Big(\frac{1}{2}\rho |\u|^{2}+\frac{1}{\gamma-1}\rho^{\gamma}\Big)dx } +\int_{U(t)}{\big(2\mu(\rho)|D\u|^{2}+\lambda(\rho)|\di \u|^{2}\big)dx}=0,
	\end{aligned}
	\end{equation}
	which implies \eqref{lem2.3 2}.
\end{proof}

If the viscosity coefficients further meet the condition that $\lambda(\rho)=2\rho \mu'(\rho)-2\mu(\rho)$, then the following further estimates are obtained. We can get the integrability of $\nabla \rho$ with respect to time and space.
\begin{lemma}\label{lemkk}
	Suppose $(\rho(x,t),\u(x,t))$ is a classical solution to System \eqref{eq1}--\eqref{initial data} in $\O \times [0,T)$ with a suitable boundary condition on $\p\O$, if $\mu(\rho)$ and $\lambda(\rho)$ satisfy \eqref{viscostiy0} and Condition (iii), then for each $0<t<T$, we have
	\begin{equation}\label{lem2.4}
	\begin{aligned}
	\frac{d}{dt}\int_{V_0(t)}{\Big{(}\frac{1}{2}\rho |\u +\nabla F(\rho)|^{2}+ \frac{1}{\gamma-1} \rho^{\gamma} \Big{)}dx}&+\int_{V_0(t)}{\nabla F(\rho) \cdot \nabla \rho^{\gamma}dx}\\
	& +\frac{1}{2}\int_{V_0{(t)}}\mu(\rho)|\nabla \u-\nabla \u^\top|^2dx=0,
	\end{aligned}
	\end{equation}
	where $V_0(t)=V(t)\backslash \{x\in V(t)|\rho(x,t)=0\}$ and $F'(\rho)= \frac{2\mu '(\rho)}{\rho}$. Moreover, 
	\begin{equation}\label{lem}
	\begin{aligned}
	\int_{0}^{t}\int_{V_0(s)}\mu '(\rho) \rho^{\gamma-2}|\nabla \rho|^2 dxds \leq E_1.
	\end{aligned}
	\end{equation}
\end{lemma}

\begin{proof}
	Multiplying $\eqref{eq1}_1$ by $F'(\rho)$ yields
	\[F_t(\rho)+\u\cdot \nabla F(\rho)+F'(\rho)\rho \di \u=0, \qquad in~  V_0(t). \]
	Operating $\nabla$ to the above equation, and setting $F'(\rho)=\frac{2 \mu '(\rho)}{\rho}$, then we obtain that
	\begin{equation}\label{Feq}
	    \nabla F_t(\rho)+\u \cdot \nabla (\nabla F(\rho))+\nabla \u\cdot \nabla F(\rho)+2\mu'(\rho)\nabla \di \u+2\nabla \mu'(\rho)\di \u=0,\quad in~  V_0(t). 
	\end{equation}
	Denote $\w:=\u+\nabla F(\rho)$, we can get from $\eqref{eq1}_2$ and \eqref{Feq} that
	\begin{equation*}
    \begin{split}
	   \rho(\w_t+\u \cdot \nabla\w)+\nabla P&-\di(2\mu(\rho)D(\u))-\nabla (\lambda(\rho)\di \u)\\
	   &+\rho\nabla \u\cdot \nabla F(\rho)+2\rho\mu'(\rho)\nabla \di \u+2\rho\nabla \mu'(\rho)\di \u= 0, \qquad in~ V_0(t).
	\end{split}
	\end{equation*}
	Multiplying the above equality by $\w$ and integrating the resulting equality over $V_0(t)$ lead to 
	\begin{equation}\label{BD}
	\begin{split}
	\frac{1}{2}\frac{d}{dt}\int_{V_0{(t)}}\rho|\w|^2dx&=-\int_{V_0{(t)}}\nabla P\cdot \w dx+ \int_{V_0{(t)}} \nabla (\lambda(\rho)\di \u)\cdot \w dx \\
	&\quad+\int_{V_0{(t)}} \di(2\mu(\rho)D(\u))\cdot \w dx-\int_{V_0{(t)}}\rho \nabla \u\cdot \nabla F(\rho)\cdot \w dx\\
	&\quad-\int_{V_0{(t)}}2\rho\nabla (\mu'(\rho)\di \u)\cdot \w dx:=\sum_{i=1}^{5} I_i.
	\end{split}
	\end{equation}
	First, since $\rho=0$ on $\p V_0(t)$, similar to \eqref{lem2.3 23}, we find that
	\begin{equation*}
	\begin{split}
	   I_1=-\frac{1}{\gamma-1}\frac{d}{dt}\int_{V_0{(t)}} \rho^\gamma dx-\int_{V_0{(t)}} \nabla F(\rho)\cdot \nabla \rho^\gamma dx.
	\end{split}
	\end{equation*}
For $I_3$, for all $1\leq i,j\leq d,$ we have	
	\begin{align}
I_3=&\int_{V_0(t)}\p_j\big[\mu(\rho)(\p_ju^i-\p_iu^j)\big]w^idx+2\int_{V_0(t)}\p_j\big[\mu(\rho)\p_iu^j\big]w^idx\nonumber\\
=&\int_{V_0(t)}\p_j\big[\mu(\rho)(\p_ju^i-\p_iu^j)\big]u^idx+\int_{V_0(t)}\p_j\big[\mu(\rho)(\p_ju^i-\p_iu^j)\big]\p_iF(\rho)dx\nonumber\\
&\quad+2\int_{V_0(t)}\big[\p_j(\mu(\rho)\p_iu^j)-\p_i(\mu(\rho)\p_ju^j)\big]w^idx+2\int_{V_0{(t)}}\p_i\big[\mu(\rho)\p_ju^j\big]w^idx\nonumber\\
=&-\frac{1}{2}\int_{V_0(t)}\mu(\rho)|\p_ju^i-\p_iu^j|^2dx+\int_{V_0(t)}\p_j\big[\mu(\rho)(\p_ju^i-\p_iu^j)\big]\p_iF(\rho)dx\nonumber\\
&\quad+2\int_{V_0(t)}\big[\p_j\mu(\rho)\p_iu^j-\p_i\mu(\rho)\p_ju^j\big]w^idx+2\int_{V_0{(t)}}\p_i\big[\mu(\rho)\p_ju^j\big]w^idx.\label{i2}
\end{align}
Notice that $\mu(\rho)=a\rho^\alpha$ for $\alpha\geq \gamma$, we have $\mu(\rho)F'(\rho)=\dfrac{2\mu(\rho)\mu'(\rho)}{\rho}=0$ on $\p V_0(t)$. Hence by approximation using smooth functions and integrating by parts on $V_0(t)$, we have 
\begin{equation}\label{FF}
\int_{V_0(t)}{\partial _{j}(\mu(\rho)\partial _{j} u ^{i})\partial _{i}F(\rho)dx}
	=\int_{V_0(t)}{\partial _{i}(\mu(\rho)\partial _{j} u ^{i})\partial _{j}F(\rho)dx},
\end{equation}
which combined with \eqref{i2} implies that
	\begin{equation*}
	\begin{split}
	I_3=-\frac{1}{2}\int_{V_0{(t)}}\mu(\rho)|\nabla \u-\nabla \u^\top|^2dx+2\int_{V_0(t)}\big[\nabla \u\cdot \nabla \mu(\rho)-\di \u\nabla\mu(\rho)+\nabla\big(\mu(\rho)\di \u\big)\big]\cdot \w dx
	\end{split}
	\end{equation*}
For $I_4$ and $I_5$, we use some calculations to discover
	\begin{equation*}
	\begin{aligned}
	I_4=&-2\int_{V_0{(t)}}\big[\nabla \u\cdot \nabla \mu(\rho)\big]\cdot \w dx,\\
	I_5=&-2\int_{V_0{(t)}}\nabla\big[\rho \mu'(\rho)\di \u\big]\cdot\w dx+2\int_{V_0{(t)}}\mu'(\rho)\di \u\nabla\rho\cdot\w dx.
	\end{aligned}
	\end{equation*}
	Inserting the above estimates of $I_i(i=1,3,4,5)$ into \eqref{BD}, we deduce that
    \begin{equation}\label{DBD}
    \begin{aligned}
    \frac{d}{dt}&\int_{V_0(t)}{\Big{(}\frac{1}{2}\rho |\u +\nabla F(\rho)|^{2}+ \frac{1}{\gamma-1} \rho^{\gamma} \Big{)}dx}+\int_{V_0(t)}{\nabla F(\rho) \cdot \nabla \rho^{\gamma}dx}\\
   =&-\frac{1}{2}\int_{V_0{(t)}}\mu(\rho)|\nabla \u-\nabla \u^\top|^2dx+\int_{V_0{(t)}} \nabla\big[\big(\lambda(\rho)-2\rho\mu'(\rho)+2\mu(\rho)\big)\di \u\big] \cdot \w dx.
   \end{aligned}
   \end{equation}
	In light of $\lambda (\rho)=2\rho \mu'(\rho)-2\mu(\rho)$, we get \eqref{lem2.4} immediately. Then substituting $F'(\rho)= \frac{2\mu '(\rho)}{\rho}$ into \eqref{lem2.4} and integrating with respect to $t$ lead to \eqref{lem}.
\end{proof}

With Lemma \ref{lemk} and \ref{lemkk} at hand, we are ready to prove the following lemma about the integrability of the viscosity coefficients with respect to time and space, which is motivated by \cite{mellet2007on}.
\begin{lemma}\label{sec}
Suppose $(\rho(x,t),\u(x,t))$ is a classical solution to System \eqref{eq1}--\eqref{initial data} in $\O \times [0,T)$ with a suitable boundary condition on $\p\O$. If $\mu(\rho)$ and $\lambda(\rho)$ satisfy \eqref{viscostiy0} and Condition (iii), then for each $0<t<T$, it holds that
	\begin{equation}\label{nxxs2}
	\begin{aligned}
	\int_{0}^{t}\int_{U(s)}{\big(2\mu(\rho)+d\lambda(\rho)\big) dxds}\leq C(t+1),
	\end{aligned}
	\end{equation}
	where $C$ is a positive constant independent of $t$.
\end{lemma}
\begin{proof}
Setting $V_1(t):=V_0(t)\backslash \mathcal{C}(t)$, where $\mathcal{C}(t):=\big\{x\in V_0(t)|~0\leq\rho(x,t)\leq 1\big\}$.  
It follows from \eqref{lem} that
	\begin{equation*}
	\begin{aligned}
	\int_{0}^{t}\int_{V_1(s)}\mu '(\rho) \rho^{\gamma-2}|\nabla \rho|^2 dxds \leq E_1.
	\end{aligned}
	\end{equation*}
	Since $\mu'(\rho)=a\alpha\rho^{\alpha-1}\geq a\alpha>0$ in $V_1(t)$, we can get that 
	\begin{equation}\label{lem2}
	\begin{aligned}
	\int_{0}^{t}\int_{V_1(s)}\rho^{\gamma-2}|\nabla \rho|^2 dxds \leq E_1.
	\end{aligned}
	\end{equation}
	In next proof, we simply mark $H^k(V_1(t))$ and $L^s(V_1(t))$ as $H^k$ and $L^s$ for $k\in \mathbb{Z}$ and $p\in \mathbb{R}^+$.
From \eqref{lem2.3 1} and \eqref{lem2}, we obtain that
	$$
	\rho \in L^{\infty}((0,t);L^{1}\cap L^{\gamma}),\;
	\nabla \rho^{\gamma / 2} \in L^{2}((0,t);L^{2}),
	$$ 
	then we get $ \rho^{\gamma / 2} \in L^{2}((0,t);H^{1}) $ and the following estimate:
\begin{equation}\label{sob}
	\int_{0}^{t}\|\rho^{\gamma/2}\|_{H^1}^2dt\leq C(t+1).
\end{equation}

	When $d=2$, it follows from Sobolev inequality that $\rho^{\gamma / 2} \in L^{2}((0,t);L^{q})$, for all $q \in [2,\infty)$. Thus $\rho^{\gamma} \in L^{1}((0,t);L^{p}) \cap L^{\infty}((0,t);L^{1})$, for all $p \in [1,\infty)$, and 
	$$\int_{0}^{t}\|\rho^{\gamma}\|_{L^p}ds\leq C(t+1),\;\sup_{0\leq s\leq t}\|\rho^{\gamma}\|_{L^1}\leq C.$$
	The interpolation inequality implies that $\rho^{\gamma} \in L^{r}((0,t) \times V_1(t))$, for all $r \in [1,2)$ with the following estimate:
	\begin{equation}\label{d2}
    \int_{0}^{t}\|\rho^{\gamma}\|_{L^r}^rds\leq C(t+1).
	\end{equation}
Notice that $0\leq\rho\leq1$ in $U(t) \backslash V_1(t)$, it follows from \eqref{d2} that
	$$
	\begin{aligned}
	\int_{0}^{t} \int_{U(s)}{\big(2\mu(\rho)+d\lambda(\rho)\big)dxds}&=2(a+b)\int_{0}^{t} \int_{U(s)}\rho^\alpha dxds\\
	&=2(a+b)\int_{0}^{t} \int_{V_1(s)}{\rho^\alpha dxds}+2(a+b)\int_{0}^{t} \int_{U(s)\backslash V_1(s)}{\rho^\alpha dxds}\\
	 &\leq C \int_{0}^{t}\int_{V_1(s)}{\rho ^{r \gamma}dxds}+Ct
	\leq C(t+1).
	\end{aligned}
	$$
	
	When $d\geq3$, Sobolev inequality and \eqref{sob} lead to $\rho^{\gamma / 2} \in L^{2}((0,t);L^{\frac{2d}{d-2}})$, which together with $\rho^{\gamma} \in L^{\infty}((0,t);L^{1})$ and the interpolation inequality gives $\rho^{\gamma} \in L^{\frac{d+2}{d}}((0,t) \times V_1(t))$, and the following estimate:
\begin{equation}\label{d3}
	\int_{0}^{t}\|\rho^{\gamma}\|_{L^{\frac{d+2}{d}}}^{\frac{d+2}{d}}ds\leq C(t+1).
\end{equation}
Notice that $0\leq\rho\leq1$ in $U(t) \backslash V_1(t)$, inequality \eqref{d3} above gives
	$$\begin{aligned}
	\int_{0}^{t} \int_{U(s)}{\big(2\mu(\rho)+d\lambda(\rho)\big)dxds}=&(2a+db)\int_{0}^{t} \int_{U(s)}\rho^\alpha dxds\\
	=&(2a+db)\int_{0}^{t} \int_{V_1(s)}{\rho^\alpha dxds}+(2a+db)\int_{0}^{t} \int_{U(s)\backslash V_1(s)}{\rho^\alpha dxds}\\
	\leq& C \int_{0}^{t}\int_{V_1(s)}{\rho ^{\frac{d+2}{d} \gamma}dxds}+Ct
	\leq C(t+1).
	\end{aligned}
	$$
Hence, we complete the proof of \eqref{nxxs2}.
\end{proof}

The following lemma is to estimate the key term in the proof of the results.
\begin{lemma}\label{kk}
	Suppose $(\rho(x,t),\u(x,t))$ is a classical solution to System \eqref{eq1}--\eqref{initial data} in $\O \times [0,T)$ with a suitable boundary condition on $\p\O$.
	If the viscosities $\mu(\rho)$ and $\lambda(\rho)$ satisfy the conditions in Theorem \ref{them1}, then for each $0<t<T$, it holds that
	\begin{equation}\label{lem2.5 1}
	\begin{aligned}
	\bigg{|}\int_{0}^{t}\int_{U(s)}{\big(2\mu(\rho)+d\lambda(\rho)\big)\di \u dxds}\bigg{|} \leq C(t+1)^{\frac{1}{2}},
	\end{aligned}
	\end{equation}
	where $C$ is a positive constant independent of $t$.
\end{lemma}

\begin{proof}
	It follows from H{\"o}lder's inequality and \eqref{lem2.3 2} that
	\begin{equation*}\label{lem2.5 11}
	\begin{aligned}
	&\bigg{|}\int_{0}^{t}\int_{U(s)}{\big(2\mu(\rho)+d\lambda(\rho)\big)\di \u dxds}\bigg{|} \\
	\leq &\int_{0}^{t}\int_{U(s)}{\big(2\mu(\rho)+d\lambda(\rho)\big) |\di \u|dxds}\\
	\leq &C\Big(\int_{0}^{t}\int_{U(s)}{\big(2\mu(\rho)+d\lambda(\rho)\big) |\di \u|^{2}dxds}\Big)^{1/2}\Big(\int_{0}^{t}\int_{U(s)}{\big(2\mu(\rho)+d\lambda(\rho)\big)dxds}\Big)^{1/2}\\
	\leq & C E_{0}^{1/2} \Big(\int_{0}^{t}\int_{U(s)}{\big(2\mu (\rho)+d \lambda (\rho)\big)dxds}\Big)^{1/2}.
	\end{aligned}
	\end{equation*}

If $\mu(\rho)$ and $\lambda(\rho)$ satisfy Condition (i), the result is obviously true.

If $\mu(\rho)$ and $\lambda(\rho)$ satisfy Condition (ii), we obtain that
	\begin{align*}\label{abc}
	\bigg{|}\int_{0}^{t}\int_{U(s)}{\big(2\mu(\rho)+d\lambda(\rho)\big)\di \u dxds}\bigg{|} 
	\leq  Ct^{1/2},
	\end{align*}
	due to
	$$
	\begin{aligned}
	\int_{U(t)}{\big(2\mu(\rho)+d\lambda(\rho)\big)dx} \leq& C \int_{U(t)}{(1+\rho ^{\gamma})dx}
	\leq C(1+E_{0}).
	\end{aligned}
	$$
	
	If $\mu(\rho)$ and $\lambda(\rho)$ satisfy Condition (iii),  from \eqref{nxxs2}, we have
	$$\begin{aligned}
	\bigg{|}\int_{0}^{t}\int_{U(s)}{\big(2\mu(\rho)+d\lambda(\rho)\big)\di \u dxds}\bigg{|} 
	\leq  C(t+1)^{1/2}.
	\end{aligned}
	$$
Hence, we finish the proof.
\end{proof}

With the estimates above at hand, we can start proving Theorem \ref{them1}.\\
\textbf{Proof of Theorem \ref{them1}.}~
	Set
	$$
	\begin{aligned}
	G(t):=\int_{U(t)}{|x|^{2}\rho(x,t)dx}.
	\end{aligned}
	$$
	From (\ref{lem2.1 k}), we get
	$$
	\begin{aligned}
	G'(t)=\frac{d}{dt}\int_{U(t)}{|x|^{2}\rho(x,t)dx}=\int_{U(t)}{\rho \u \cdot \nabla |x|^{2}dx}=2\int_{U(t)}{\rho \u \cdot x dx},
	\end{aligned}
	$$
	and
	$$
	\begin{aligned}
	G''(t)=& 2\frac{d}{dt}\int_{U(t)}{\rho \u \cdot x dx}=2\int_{U(t)}{\rho [(\u \cdot x)_{t}+\u \cdot \nabla (\u \cdot x)]dx}\\
	=& 2\int_{U(t)}{\rho (\u_{t}+\u \cdot \nabla \u)\cdot x dx}+2\int_{U(t)}{\rho \u \cdot(\u \cdot \nabla x)dx}\\
	=& 2\int_{U(t)}{\big(\di (2\mu(\rho)D \u)+\nabla (\lambda(\rho)\di \u)-\nabla P\big)\cdot x dx}+2\int_{U(t)}{\rho |\u|^{2}dx}.
	\end{aligned}
	$$
Notice that $\rho=0$ on $\p U(t),$ integrating by parts leads to
	$$
	\begin{aligned}
	\int_{U(t)}{\nabla P \cdot x dx}=-\int_{U(t)}{P \di x dx}=-d\int_{U(t)}{Pdx},
	\end{aligned}
	$$
	$$
	\begin{aligned}
	\int_{U(t)}{\big(\di (2\mu(\rho)D \u)+\nabla (\lambda(\rho)\di \u)\big)\cdot x dx}=-\int_{U(t)}{\big(2\mu(\rho)+d\lambda(\rho)\big)\di \u dx},
	\end{aligned}
	$$
	so we conclude that
	\begin{equation}\label{them1 1}
	\begin{aligned}
	G''(t)=&2\int_{U(t)}{\rho |\u|^{2}dx}+2d \int_{U(t)}{P dx}-2 \int_{U(t)}{(2\mu(\rho)+d \lambda(\rho))\di \u dx}.
	\end{aligned}
	\end{equation}
  Equality \eqref{lem2.3 1} and Jensen inequality imply
\begin{equation}\label{them1 2}	
	\begin{aligned}
   \int_{U(t)}{Pdx}=\int_{U(t)}{\rho^{\gamma}dx} 
	\geq |U(t)|^{1-\gamma}\bigg{(}\int_{U(t)}{\rho dx}\bigg{)}^{\gamma}
	=|U(t)|^{1-\gamma}M_{0}^{\gamma}.
	\end{aligned}
\end{equation}	
From Lemma \ref{kk}, we have that
\begin{equation}\label{them1 3}	
	\begin{aligned}
	\bigg{|}\int_{0}^{t}{\int_{U(s)}{(2\mu(\rho)+d \lambda(\rho))\di \u dx}ds}\bigg{|}\leq C(t+1)^{1/2}.
	\end{aligned}
\end{equation}
	Then, it follows from \eqref{them1 1},\eqref{them1 2} and \eqref{them1 3} immediately that 
	$$
	\begin{aligned}
	G'(t)\geq G_{1}+2d |U(t)|^{1-\gamma}M_{0}^{\gamma}t-C(t+1)^{1/2}.
	\end{aligned}
	$$
	Therefore, for each $0<t<T$, we have
	$$
	\begin{aligned}
	G(t)\geq G_{0}+G_{1}t+C M_{0}^{\gamma}t^{2}-C(t+1)^{3/2}.
	\end{aligned}
	$$
	On the other hand, by virtue of the mass conservation equation and the fact that $U(t)$ is always in a bounded domain $\O\subset B_R(0)$ for some $R>0$, we obtain that
	$$
	\begin{aligned}
	G(t)=\int_{U(t)}{|x|^{2}\rho(x,t)dx}\leq R^{2}\int_{U(t)}{\rho(x,t)dx}=M_{0}R^{2}.
	\end{aligned}
	$$
 This leads to a finite bound on the life span of the classical solutions. Furthermore, the maximal time of existence satisfies $T\leq T^*$, where $T^*$ is the positive root of the equation $$CM_{0}^{\gamma}t^{2}-C(t+1)^{3/2}+G_{1}t+G_{0}-M_{0}R^{2}=0.$$ Therefore, the proof of Theorem \ref{them1} is completed.

\section*{Acknowledgements}
  Rongfeng Yu is partially supported by the Fundamental Research Funds for the Central Universities of China (Grant No.19lgpy237), and Natural Science Foundation of Guangdong Province, China (Grant No.2020B1515310004), Zheng-an Yao is partially supported by National Natural Science Foundation of China (Grant No.11971496).


\addcontentsline{toc}{section}{\refname}
\begin{thebibliography}{99}

\bibitem{bian2019finite}
D.~Bian and J.~Li.
\newblock Finite time blow up of compressible {N}avier-{S}tokes equations on
half space or outside a fixed ball.
\newblock {\em J. Differential Equations}, 267(12):7047--7063, 2019.

\bibitem{bresch2003existence}
D.~Bresch and B.~Desjardins.
\newblock Existence of global weak solutions for a 2{D} viscous shallow water
equations and convergence to the quasi-geostrophic model.
\newblock {\em Comm. Math. Phys.}, 238(1):211--223, 2003.

\bibitem{bresch2007on}
D.~Bresch, B.~Desjardins, and D.~G\'{e}rard-Varet.
\newblock On compressible {N}avier-{S}tokes equations with density dependent
viscosities in bounded domains.
\newblock {\em J. Math. Pures Appl.}, 87(2):227--235, 2007.

\bibitem{cho2006blow-up}
Y.~Cho and B.~J. Jin.
\newblock Blow-up of viscous heat-conducting compressible flows.
\newblock {\em J. Math. Anal. Appl.}, 320(2):819--826, 2006.

\bibitem{2011Ding}
S.~Ding, H.~Wen, and C.~Zhu.
\newblock Global classical large solutions to 1{D} compressible
{N}avier-{S}tokes equations with density-dependent viscosity and vacuum.
\newblock {\em J. Differential. Equations}, 251(6):1696--1725, 2011.

\bibitem{duan2019finite-time}
B.~Duan, Z.~Luo, and W.~Yan.
\newblock Finite-time blow-up of classical solutions to the rotating shallow
water system with degenerate viscosity.
\newblock {\em Z. Angew. Math. Phys.}, 70(2):1--9, 2019.

\bibitem{feireisl2001on}
E.~Feireisl, A.~Novotný, and H.~Petzeltov\'{a}.
\newblock On the existence of globally defined weak solutions to the
{N}avier-{S}tokes equations.
\newblock {\em J. Math. Fluid Mech.}, 3(4):358--392, 2001.

\bibitem{guo2008spherically}
Z.~Guo, Q.~Jiu, and Z.~Xin.
\newblock Spherically symmetric isentropic compressible flows with
density-dependent viscosity coefficients.
\newblock {\em SIAM J. Math. Anal.}, 39(5):1402--1427, 2008.

\bibitem{2016Huang}
X.~Huang and J.~Li.
\newblock Existence and blowup behavior of global strong solutions to the
two-dimensional barotrpic compressible {N}avier–{S}tokes system with vaccum
and large data.
\newblock {\em J. Math. Pures Appl.}, 106(1):123--154, 2016.

\bibitem{huang2012global}
X.~Huang, J.~Li, and Z.~Xin.
\newblock Global well-posedness of classical solutions with large oscillations
and vacuum to the three-dimensional isentropic compressible {N}avier-{S}tokes
equations.
\newblock {\em Comm. Pure. Appl. Math.}, 65(4):549--585, 2012.

\bibitem{jiang2001on}
S.~Jiang and P.~Zhang.
\newblock On spherically symmetric solutions of the compressible isentropic
{N}avier-{S}tokes equations.
\newblock {\em Comm. Math. Phys.}, 215(3):559--581, 2001.

\bibitem{jiu2014global}
Q.~Jiu, Y.~Wang, and Z.~Xin.
\newblock Global well-posedness of 2{D} compressible {N}avier-{S}tokes
equations with large data and vacuum.
\newblock {\em J. Math. Fluid Mech.}, 16(3):483--521, 2014.

\bibitem{jiu2018}
Q.~Jiu, Y.~Wang, and Z.~Xin.
\newblock Global classical solution to two-dimensional compressible
{N}avier-{S}tokes equations with large data in {$\mathbb{R}^2$}.
\newblock {\em Phys. D}, 376/377:180--194, 2018.

\bibitem{li2008vanishing}
H.-L. Li, J.~Li, and Z.~Xin.
\newblock Vanishing of vacuum states and blow-up phenomena of the compressible
{N}avier-{S}tokes equations.
\newblock {\em Comm. Math. Phys.}, 281(2):401--444, 2008.

\bibitem{2019liwang}
H.-L. Li, Y.~Wang, and Z.~Xin.
\newblock Non-existence of classical solutions with finite energy to the
{C}auchy problem of the compressible {N}avier-{S}tokes equations.
\newblock {\em Arch. Ration. Mech. Anal}, 232(2):557--590, 2019.

\bibitem{Liliang2016}
J.~Li and Z.~Liang.
\newblock Some uniform estimates and large-time behavior of solutions to
one-dimensional compressible {N}avier-{S}tokes system in unbounded domains
with large data.
\newblock {\em Arch. Ration. Mech. Anal}, 220(3):1195--1208, 2016.

\bibitem{jing2015global}
J.~Li and Z.~Xin.
\newblock Global existence of weak solutions to the barotropic compressible
{N}avier-{S}tokes flows with degenerate viscosities.
\newblock {\em arXiv: 1504.06826v2}, 2015.

\bibitem{li2019global}
J.~Li and Z.~Xin.
\newblock Global well-posedness and large time asymptotic behavior of classical
solutions to the compressible {N}avier-{S}tokes equations with vacuum.
\newblock {\em Ann. PDE}, 5(1):7, 2019.

\bibitem{Li2016Recent}
Y.~Li, R.~Pan, and S.~Zhu.
\newblock Recent progress on classical solutions for compressible isentropic
{N}avier-{S}tokes equations with degenerate viscosities and vacuum.
\newblock {\em Bull. Braz. Math. Soc.(N. S.)}, 47(2):507--519, 2016.

\bibitem{2015On}
Y.~Li, R.~Pan, and S.~Zhu.
\newblock On classical solutions for viscous polytropic fluids with degenerate
viscosities and vacuum.
\newblock {\em Arch. Ration. Mech. Anal.}, 234(3):1281--1334, 2019.

\bibitem{Lion1}
P.-L. Lions.
\newblock {\em Mathematical topics in fluid mechanics. Vol.1. Incompressible
	models}.
\newblock Oxford University Press, New York, 1996.

\bibitem{Lion}
P.-L. Lions.
\newblock {\em Mathematical topics in fluid mechanics. Vol.2. Compressible
	models}.
\newblock Oxford University Press, New York, 1998.

\bibitem{liu1998}
T.-P. Liu, Z.~Xin, and T.~Yang.
\newblock Vaccum states for compressible flow.
\newblock {\em Discrete Contin. Dynam. systems}, 4(1):1--32, 1998.

\bibitem{2019luozhou}
Z.~Luo and Y.~Zhou.
\newblock On the existence of local classical solutions to the
{N}avier-{S}tokes equations with degenerate viscosities.
\newblock {\em J. Math. Phys}, 60(9):091508, 2019.

\bibitem{Majda}
A.~J. Majda and A.~L. Bertozzi.
\newblock {\em Vorticity and incompressible flow}.
\newblock Combridge,University Press, Combridge,University Press England, 2002.

\bibitem{matsumura1983initial-boundary}
A.~Matsumura and T.~Nishida.
\newblock Initial-boundary value problems for the equations of motion of
compressible viscous and heat-conductive fluids.
\newblock {\em Comm. Math. Phys.}, 89(4):445--464, 1983.

\bibitem{mellet2007on}
A.~Mellet and A.~Vasseur.
\newblock On the barotropic compressible {N}avier-{S}tokes equations.
\newblock {\em Comm. Partial Differential Equations}, 32(3):431--452, 2007.

\bibitem{Qinyao2010}
X.~Qin and Z.-a. Yao.
\newblock Global solutions of the free boundary problem for the compressible
{N}avier-{S}tokes equations with density-dependent viscosity.
\newblock {\em Commun. Pure Appl. Anal}, 9(4):1041--1052, 2010.

\bibitem{rozanova2008blow}
O.~Rozanova.
\newblock Blow up of smooth highly decreasing at infinity solutions to the
compressible {N}avier-{S}tokes equations.
\newblock {\em J. Differential Equations}, 245(7):1762--1774, 2008.

\bibitem{1995Kazhikhov}
V.~A. Vaigant and A.~V. Kazhikhov.
\newblock On existence of global solutions to the two-dimensional
{N}avier-{S}tokes equations for a compressible viscous fluid.
\newblock {\em Siberian Math. J.}, 36(6):1108--1141, 1995.

\bibitem{vasseur2016existence}
A.~Vasseur and C.~Yu.
\newblock Existence of global weak solutions for 3{D} degenerate compressible
{N}avier-{S}tokes equations.
\newblock {\em Invent. Math.}, 206(3):935--974, 2016.

\bibitem{xin1998blowup}
Z.~Xin.
\newblock Blowup of smooth solutions to the compressible {N}avier-{S}tokes
equation with compact density.
\newblock {\em Comm. Pure Appl. Math.}, 51(3):229--240, 1998.

\bibitem{xin2013on}
Z.~Xin and W.~Yan.
\newblock On blowup of classical solutions to the compressible
{N}avier-{S}tokes equations.
\newblock {\em Comm. Math. Phys.}, 321(2):529--541, 2013.

\bibitem{xin2019global}
Z.~Xin and S.~Zhu.
\newblock Global well-posedness of regular solutions to the three-dimensional
isentropic compressible navier-stokes equations with degenerate viscosities
and vacuum.
\newblock {\em arXiv:1806.02383v2}, 2019.

\bibitem{Yangyao2001}
T.~Yang, Z.-a. Yao, and C.~Zhu.
\newblock Compressible {N}avier-{S}tokes equations with density-dependent
viscosity and vacuum.
\newblock {\em Comm. Partial Differential Equations}, 26(5--6):965--981, 2001.

\bibitem{yang2002compressible}
T.~Yang and C.~Zhu.
\newblock Compressible {N}avier-{S}tokes equations with degenerate viscosity
coefficient and vacuum.
\newblock {\em Comm. Math. Phys.}, 230(2):329--363, 2002.

\bibitem{yu}
R.~Yu.
\newblock Expansion of a compressible non‐barotropic fluid in vacuum.
\newblock {\em Math. Meth. Appl. Sci.}, 44(5):3521--3526, 2021.

\end{thebibliography}

\phantomsection

\end{document}